\documentclass[12pt]{article} 

\usepackage{amsmath}
\usepackage{amsthm}
\usepackage{amsfonts}
\usepackage{mathrsfs}
\usepackage{stmaryrd}
\usepackage{setspace}
\usepackage{fullpage}
\usepackage{amssymb}
\usepackage{breqn}
\usepackage{enumitem}
\usepackage{bbold} 
\usepackage{authblk}
\usepackage{comment}
\usepackage{hyperref}
\usepackage{pgf,tikz}
\usepackage{graphicx}
\usetikzlibrary{decorations.pathreplacing,arrows}
\tikzstyle{vertex}=[circle,draw=black,fill=black,inner sep=0,minimum size=3pt,text=white,font=\footnotesize]

\newtheorem{thm}{Theorem}

\newtheorem{prop}[thm]{Proposition}

\newtheorem{definition}[thm]{Definition}

\newtheorem{cor}[thm]{Corollary}

\newtheorem{rem}[thm]{Remark}

\newcommand\cA{{\mathcal A}}

\newcommand\cF{{\mathcal F}}
\newcommand\cG{{\mathcal G}}
\newcommand\cH{{\mathcal H}}

\newcommand{\abs}[1]{\left\lvert{#1}\right\rvert}

\newcommand{\ignore}[1]{}

\title{
	Generalized Tur\'an problems for disjoint copies of graphs
}

\begin{document}
	
	\author{
		D\'aniel Gerbner \thanks{Alfr\'ed R\'enyi Institute of Mathematics, Hungarian Academy of Sciences. e-mail: gerbner@renyi.hu}
		\qquad
		Abhishek Methuku \thanks{ Central European University, Budapest. e-mail: abhishekmethuku@gmail.com}
		\qquad
		M\'at\'e Vizer \thanks{Alfr\'ed R\'enyi Institute of Mathematics, Hungarian Academy of Sciences. e-mail: vizermate@gmail.com.}}
	
	\date{
		\today}

	\maketitle
	
	\begin{abstract}
		Given two graphs $H$ and $F$, the maximum possible number of copies of $H$ in an $F$-free graph on $n$ vertices is denoted by $ex(n,H,F)$. We investigate the function $ex(n,H,kF)$, where $kF$ denotes $k$ vertex disjoint copies of a fixed graph $F$. Our results include cases when $F$ is a complete graph, cycle or a complete bipartite graph. 
		
	\end{abstract}
	
	\vspace{4mm}
	
	\noindent
	{\bf Keywords:} Tur\'an numbers, disjoint copies, generalized Tur\'an
	
	\noindent
	{\bf AMS Subj.\ Class.\ (2010)}: 05C35, 05C38
	
	\section{Introduction}
	
	The vertex set of a graph $G$ is denoted by $V(G)$ and its edge set is denoted by $E(G)$. The \textit{disjoint union} $G \cup H$ of graphs $G$ and $H$ with disjoint vertex sets $V(G)$ and $V(H)$ is the graph with the vertex set $V(G) \cup V(H)$ and edge set $E(G) \cup E(H)$. The \textit{join} $G + H$, of graphs $G$ and $H$ with disjoint vertex sets is the graph obtained by taking a copy of $G$ and a copy of $H$ on disjoint vertex sets and adding all the edges between them. 
    
    Given a positive integer $k$ and a graph $F$, the vertex disjoint union of $k$ copies of the graph $F$ is denoted by $kF$. Let $C_l$ denote a cycle of length $l$, $K_{s,t}$ denote the complete bipartite graph with parts of sizes $s$ and $t$ and let $K_r$ denote the complete graph on $r$ vertices. 
	
	For a set of graphs $\cF$ the \textit{Tur\'an number} of  $\cF$, $ex(n,\cF)$, denotes the maximum number of edges of an $n$-vertex graph having no member of $\cF$ as a subgraph. If $\cF$ contains only a single graph $F$, we simply denote it by $ex(n,F)$. This function has been intensively studied, starting with the theorems of Mantel \cite{M1907} and Tur\'an \cite{T1941} that determine $ex(n,K_{r+1})$ for $r \ge 3$. Tur\'an also showed in \cite{T1941} that a complete $r$-partite graph on $n$ vertices with as equal parts as possible is the unique extremal graph. This extremal graph is called \textit{Tur\'an graph} and is denoted by $T_r(n)$. See, for example, \cite{FS2013,S1997} for surveys on this topic.
    
    Simonovits \cite{S1966} and Moon \cite{moon} showed that if $n$ is sufficiently large, then $K_{k-1} + T_r(n - k + 1)$ is the unique extremal graph for $\cF=\{ kK_{r+1} \}$. In \cite{G2011} Gorgol initiated the systematic investigation of Tur\'an numbers of disjoint copies of connected graphs and proved the following.
	\begin{thm}[Gorgol]		\label{Gorgol}
		For every non-empty graph $F$ and $k \ge 1$, we have $$ex(n,kF)=ex(n,F)+O(n).$$ 
	\end{thm}
	\noindent
	
In fact, Gorgol proved the following sharper upper bound: If $F$ is an arbitrary connected graph and $k$ is an arbitrary positive integer, then $ex(n,kF) \le ex(n-(k-1)|V(F)|,F)+ \binom{(k-1)|V(F)|}{2} +(k-1) |V(F)|(n- (k - 1)|V(F)|)$ for $n \ge k|V(F)|$. For recent results about Tur\'an numbers of disjoint copies of graphs see \cite{BK2011,LLP2013,YZ2017}.
	
	\
	
	Given a graph $H$ and a set of graphs $\cF$, the maximum possible number of copies of $H$ in an $n$-vertex graph that does not contain any copy of $F \in \cF$ is denoted by $ex(n,H,\cF)$ and is called \textit{Generalized Tur\'an number}. If $\cF=\{F\}$, we simply denote it by $ex(n,H,F)$. Note that $ex(n, K_2, F) = ex (n, F)$. Erd\H os \cite{E1962} determined $ex(n,K_s,K_t)$ exactly. We will later use the following consequence of his result.
	
	\begin{prop}[Erd\H os]\label{Erdosprop} For $s < t$ we have:
		
		$$ex(n,K_s,K_t)= (1+o(1))\binom{t-1}{s}\left (\frac{n}{t-1} \right)^s.$$
		
	\end{prop}
	
	Another notable result is that of Bollob\'as and Gy\H{o}ri \cite{BG2008}, who showed that $ex(n,K_3, C_5) =\Theta(n^{3/2})$. The systematic study of the function $ex(n,H,F)$ was initiated by Alon and Shikhelman in \cite{ALS2016}.
	
	The function $ex(n,H,F)$ is closely related to the area of Berge hypergraphs. A \textit{Berge cycle} of length $k$ is an alternating sequence of distinct vertices and hyperedges of the form $v_1$,$h_{1}$,$v_2$,$h_{2},\ldots,v_k$,$h_{k}$,$v_1$ where $v_i,v_{i+1} \in h_{i}$ for each $i \in \{1,2,\ldots,k-1\}$ and $v_k,v_1 \in h_{k}$ and is denoted by Berge-$C_k$. Gy\H{o}ri and Lemons \cite{Gy_Lemons} showed that any $r$-uniform hypergraph avoiding a Berge-$C_{2l+1}$ contains $O(n^{1+1/l})$ hyperedges. They also showed that any $r$-uniform hypergraph avoiding a Berge-$C_{2l}$ contains $O(n^{1+1/l})$ hyperedges. These results easily imply the following. 
	
	\begin{thm}
		\label{GyLemonsthm}
        We have
        \begin{enumerate}
        \item[(a)] For any $r \ge 3, \ l  \ge 2$, we have
		$$ex(n, K_r, C_{2l+1}) = O(n^{1+1/l}).$$
        
        \item[(b)] For any $r , l  \ge 2$, we have
		$$ex(n, K_r, C_{2l}) = O(n^{1+1/l}).$$
        \end{enumerate}
		
	\end{thm}
	
    \begin{proof}
    We will prove (a) and (b) simultaneously. Let $G$ be a $C_{k}$-free graph. Replace each clique of size $r$ in it with a hyperedge on the same vertex set as the clique. It is easy to see that the resulting $r$-uniform hypergraph $H$ does not contain a Berge-$C_{k}$, and in both cases $k = 2l$ and $k = 2l+1$, $H$ has at most $O(n^{1+1/l})$ hyperedges by the theorem of Gy\H{o}ri and Lemons \cite{Gy_Lemons} mentioned before. This completes the proof as the number of cliques in $G$ is equal to the number of hyperedges in $H$.
    \end{proof}

	

	Alon and Shikhelman \cite{ALS2016} noted that while $ex(n,K_3,C_5)=\Theta(n^{3/2})$, we have $ex(n,K_3,2C_5)=\Theta(n^2)$, showing that $ex(n,H,F)$ and $ex(n,H,kF)$ can have different order of magnitudes, unlike the graph case, where $ex(n,kF)=\Theta(ex(n,F))$ (see Theorem \ref{Gorgol}).
	
	Our goal in this paper is to explore this phenomenon. Most of our theorems will relate $ex(n,H,kF)$ to $ex(n,H,F)$ for several graphs $H$ and $F$. 
	
	\subsection*{General approach}
	
	The most typical example of a $kF$-free graph is obtained by taking an $F$-free graph $G$ on $n-k+1$ vertices, and considering $K_{k-1} + G$. (We will sometimes refer to these $k-1$ vertices of degree $n-1$ in $K_{k-1}$ as \emph{universal} vertices of $K_{k-1} + G$.) Indeed, since any copy of $F$ in $K_{k-1} + G$ must contain a vertex of $K_{k-1}$ and as there are only $k-1$ vertices in $K_{k-1}$, it is impossible to find $k$ vertex disjoint copies of $F$.
	
	For example, let us take a $C_5$-free graph $G$ on $n-1$ vertices, add a new vertex $v$ and connect it to all the vertices of $G$. This graph shows $ex(n,K_3,2C_5)=\Omega(n^2)$. In addition to the triangles in $G$ (which are at most $O(n^{3/2})$ by the result of Bollob\'as and Gy\H ori \cite{BG2008} mentioned before), there are triangles which contain  $v$ and an edge of $G$. If $G$ is the Tur\'an graph $T_2(n-1)$, then there are $\Omega(n^2)$ many such triangles. What happens here is that instead of counting the copies of $K_3$ in a $C_5$-free graph, we count the copies of $K_2$ (which is a subgraph of $K_3$). As this happens to be of larger order of magnitude, we get more copies of $K_3$ in a $2C_5$-free graph than in a $C_5$-free graph. 
	
	\smallskip
	
	To prove the upper bounds we will need the following operation: for an integer $k$, a graph $F$ and a $kF$-free graph $G$, we consider the maximum number of disjoint copies of $F$ in $G$. In the rest of the paper we denote the subgraph of $G$ consisting of these copies by $G_L$, and the set of vertices spanned by $G_L$ is denoted by $L(G)$.
We denote by $R(G)$ the set $V(G)\setminus L(G)$ of the remaining vertices, and by $G_R$ the subgraph of $G$ induced by them. We call this partition of the vertices a \textit{canonical $(k,F)$-partition} of $G$. (If it is clear from the context we simply write canonical partition.) Note that $G_R$ is $F$-free.
	
	\
	
	\subsection*{Structure of the paper}
	The rest of this paper is divided into sections based on which graph is forbidden. In Section 2 we prove bounds on $ex(n, H, kF)$ for general $F$, while in Section 3 one of our main results is to determine the order of magnitude of $ex(n,K_s,kK_t)$ for all $s \ge t \ge 2$ and $k \ge 1$. In Section 4, we obtain bounds on $ex(n, K_r, kC_{l})$. In Section 5 we study the case when $F$ is a complete bipartite graph. We finish our article with some concluding remarks and open problems in Section 6.
	
	\section{Forbidding a general F}
	
	\subsection{Counting arbitrary graphs}
	
    For a family of graphs $\cH$, let us define $N(\cH,G)$ as the number of copies of members of $\cH$ in $G$. If $\cH = \{H\}$, then we simply write $N(H,G)$ instead of $N(\cH,G)$.
    
    \vspace{2mm}
    
Let $\cH^{ind}$ be the family of all induced subgraphs of a graph $H$. 
Let $$\overline{ex}(n,H,F) := \max \{N(\cH^{ind},G) : G \text{ is an $F$-free graph on $n$ vertices}\}.$$ 
	
	\begin{rem}
		\label{alphaH}
		Note that if $F$ is a non-empty graph (i.e., contains at least one edge), then 
		$$\overline{ex}(n,H,F)\ge \binom{n}{\alpha(H)}.$$
		Indeed, let $G$ be an $F$-free graph on $n$ vertices and let $I \in \cH$ be an induced subgraph spanned by a largest independent set of $H$. Then any set of $\alpha(H)$ vertices in $G$ forms a copy of $I$.
	\end{rem}

	\begin{thm}\label{arbit} For any $k \ge 2$ we have,
		$$ex(n,H,kF)=O(\overline{ex}(n,H,F)).$$ Moreover, if $k \ge |V(H)|$, then $$ex(n,H,kF)=\Theta(\overline{ex}(n,H,F)).$$
	\end{thm}
	
	\begin{proof} For the lower bound, we take an $F$-free graph $G$ on $n-k+1$ vertices that contains $\overline{ex}(n-k+1,H,F)$ copies of induced subgraphs of $H$. Then $K_{k-1}+G$ is obviously $kF$-free. If $k \ge |V(H)|$, then every copy of an induced subgraph (having at least one vertex) of $H$ 
    in $G$ can be extended to a copy of $H$ in $K_{k-1}+G$, using the vertices of $K_{k-1}$. (A small technical issue is the following: Let $Z$ be the induced subgraph of $H$ with zero vertices. A copy of $Z$ in $G$ cannot be extended to a copy of $H$ in $K_{k-1}+G$ if $k = \abs{V(H)}$, but there is only one copy of $Z$ in $G$.) Thus $K_{k-1}+G$ contains at least $\overline{ex}(n-k+1,H,F)-1$ copies of $H$. Now using the following standard argument, we conclude that $\overline{ex}(n-k+1,H,F)-1 = \Omega(\overline{ex}(n,H,F))$: Consider a graph $G$ on $n$ vertices with $\overline{ex}(n,H,F)$ copies of induced subgraphs of $H$. Then a subgraph of $G$ induced by a random subset of vertices of size $n-k+1$, contains at least $(1+o(1)) \overline{ex}(n,H,F)$ copies of induced subgraphs of $H$. On the other hand, this subgraph contains at most $\overline{ex}(n-k+1,H,F)$ copies of induced subgraphs of $H$. This finishes the proof of the lower bound.
		
		Now we continue with the upper bound. Let us consider a $kF$-free graph $G$, and its canonical partition. Then every copy of $H$ in $G$ contains a subgraph in $G_R$, which contains an induced subgraph of $H$ (note that this subgraph may have zero vertices). Moreover, each copy of an induced subgraph of $H$ in $G_R$ can be extended to a copy of $H$ in $G$ using vertices from $L(G)$ in at most $2^{|L(G)|}=O(1)$ ways. Therefore, the number of copies of $H$ in $G$ is at most $O(\overline{ex}(n,H,F))$, as desired. \end{proof}
	
	\subsection{Counting triangles}
	
	Let $F_1, \dots, F_k$ be graphs different from $K_2$ 
 and let $F$ be their vertex-disjoint union. (Note that the $F_i$'s are not necessarily different.)    

	\begin{thm} 
		\label{general_bound}
       \begin{equation*}\label{folso} ex(n,K_3,F)=\Theta\left(\max_{1 \le i \le k} \{ex(n,K_3,F_i)\}+\max_{1 \le i < j\le k} ex(n,\{F_i,F_j\})\right).\end{equation*}
       

		
	\end{thm}
    

	\begin{proof}[Proof of Theorem \ref{general_bound}]

		For the lower bound, consider the following two constructions. 
		\begin{enumerate}
			\item Take an $F_i$-free graph containing the largest number of triangles. This graph is obviously $F$-free, showing that $ex(n,K_3,F)\ge \max_{1 \le i \le k} \{ex(n,K_3,F_i)\}$.
			
			\item Now let $i,j$ be two integers with $1 \le i < j \le k$. Take an $\{F_i,F_j\}$-free graph $G_0$ on $n-1$ vertices, and add a universal vertex $v$. Then any copy of $F_i$ and any copy of $F_j$ in the resulting graph $G$ must contain $v$, thus there are no vertex-disjoint copies of $F_i$ and $F_j$ in this graph. Therefore it does not contain $F$. Furthermore, the number of triangles in $G$ is at least the number of edges in $G_0$, as these edges form a triangle with $v$. Thus we have $ex(n,K_3,F) \ge \max_{1 \le i < j\le k} ex(n,\{F_i,F_j\})$.
		\end{enumerate}
		
		For the upper bound, we use induction on $k$. The base case $k=1$ is trivial.
        Let $F'$ be the graph obtained by deleting $F_i$ from $F$ and $F''$ be the graph obtained by deleting $F_j$ from $F$. Let us consider an $F$-free graph $G$ on $n$ vertices. If $G$ is $F'$-free or $F''$-free, then we are done by induction. Thus we may assume $G$ contains a copy of $F'$ and a copy of $F''$, and let $L$ be the union of their vertex sets. Note that these copies share at least one vertex, otherwise there would be a copy of $F$ in $G$.
        
        Let $G'$ be the subgraph of $G$ induced by $V(G) \setminus L$. Then $G'$ is obviously both $F_i$-free and $F_j$-free, hence it contains at most $ex(n,\{F_i,F_j\})$ edges and at most $ex(n,K_3,F_i)$ triangles. Let $T_s$ denote the set of triangles in $G$ which contain exactly $s$ vertices from $L$. Then we have $|T_3|=O(1)$, $|T_2|=O(n)$, $|T_1|=O(ex(n,\{F_i,F_j\}))$ and $|T_0| \le ex(n,K_3,F_i)$. Adding up these bounds, the proof is complete.

	\end{proof}

\begin{rem}

\begin{itemize}
\item Note that 
by Theorem \ref{general_bound}, we have $ex(n,K_3,kF)= \Theta(ex(n,K_3,F)+ex(n,F))$, for any integer $k \ge 2$. This shows that when $k$ increases from 1 to 2, the order of magnitude of $ex(n,K_3,kF)$ can increase, but from then on (i.e., for $k \ge 2$), there is no further increase.

\item The Compactness conjecture of Erd\H os and Simonovits \cite{ES1982} states that for any finite family $\cG$ of graphs  there is a $G\in \cG$ such that $ex(n,\cG)=\Theta(ex(n,G))$. It is known to be true for several classes of graphs, for example if $\cG$ contains at most one bipartite graph.

Let $ex^{sec}(n, F)$ be the second largest of the Tur\'an numbers $ex(n, F_i)$, $1 \le i \le k$. Note that if the Compactness conjecture is true (even if it is true only for families of two graphs), then $\max_{1 \le i < j\le k} ex(n,\{F_i,F_j\}) = \Theta(ex^{sec}(n, F))$. (In particular, if $F$ is non-bipartite, then this is the case.) Thus if the Compactness conjecture is true, then Theorem \ref{general_bound} can be stated as follows:

$$ex(n,K_3,F)=\Theta\left(\max_{1 \le i \le k} \{ex(n,K_3,F_i)\}+ex^{sec}(n, F)\right).$$

\end{itemize}





\end{rem}

	\noindent
	
	\begin{definition}
		Let us define 
		$$ex^*(n,F) := \max_{G} \{(k-1) \abs{E(G)} + N(K_3,G) : G \text{ is an $n$-vertex $F$-free graph}\}.$$ 
	\end{definition}
    
	Note that we have $ex(n,K_3,F) \le ex^*(n,F) \le (k-1) ex(n,F)+ex(n,K_3,F)$. Let us consider an arbitrary $F$, and let $F_u$ be the graph we get by deleting the vertex $u$ from $F$.
    
	
	\begin{thm}\label{2copy} Let $|V(F)|\ge 4$. Then for every $u\in V(F)$ we have $$ex^*(n-k+1,F)\le ex(n,K_3,kF) \le ex^*(n,F)+(k-1)|V(F)|ex(n,F_u)+O(n).$$ 
	\end{thm}
	
	\begin{proof}
		
		For the lower bound of $ex(n,K_3,kF)$, take an $F$-free graph $G$ on $n-k+1$ vertices for which $(k-1) \abs{E(G)} + N(K_3,G)$ is maximum, and consider $K_{k-1}+G$. Then every edge of $G$, together with the $k-1$ universal vertices, gives $k-1$ triangles. This shows $ex(n,K_3,kF)\ge ex^*(n-k+1,F)$. 
		
		For the upper bound we use induction on $k$, the case $k=1$ is trivial. If a $kF$-free graph $G$ does not contain $(k-1)F$, then we are done by induction. So we may assume $G$ contains a copy of $(k-1)F$.
		
		Recall that $F_u$ is the graph obtained by removing the vertex $u$ from $F$. Let $X(u)$ be the set of neighbors of $u$ in $F$. Let $F^*$ be the graph we get by taking $(k-1)|V(F)|+1$ vertex disjoint copies of $F_u$ and an additional vertex $v$ (we call it the \textit{center} of $F^*$), that is connected to all the vertices in $X(u)$ in each copy of $F_u$. In other words we take $(k-1)|V(F)|+1$ vertex-disjoint copies of $F$ and identify the vertices $u$ from each copy of $F$. So $F^*$ is \emph{created} by $(k-1)|V(F)|+1$ copies of $F$ that intersect only in the center. Observe that if we are given a copy $F_0$ of $F$ and a copy of $F^*$, such that $F_0$ does not contain the center of the $F^*$, then at least one of the copies of $F$ which create $F^*$ is disjoint from $F_0$.
		
		
		Now let us assume $G$ contains $k-1$ vertex-disjoint copies of $F^*$ and let $v_1,\dots, v_{k-1}$ be their centers. In that case every copy $F_0$ of $F$ contains at least one of the centers. Indeed, otherwise each of the $k-1$ copies of $F^*$ contain a copy of $F$ that is disjoint from $F_0$, and these copies are all disjoint from each other. Therefore they form a copy of $kF$, a contradiction. Thus by deleting $v_1, \dots, v_{k-1}$ we get an $F$-free graph $G'$, and $ex(n,K_3,kF)$ is at most the number of triangles in $G'$ plus the number of triangles $t_i$ with exactly $i$ vertices from $\{v_1, \dots, v_{k-1}\}$ for each $1 \le i \le 3$. $t_1$ is at most  $k-1$ times the number of edges in $G'$, whereas $t_2 + t_3 = O(n)$. Therefore, the total number of triangles in $G$ is at most $ex^*(n-k+1,F) + O(n)$, as desired. 
		
		So we may assume $G$ does not contain $(k-1)F^*$. Let us consider the canonical partition, i.e. a copy $G_L$ of $(k-1)F$ and  the graph $G_R$ induced by the remaining vertices. Then there are $O(n)$ triangles containing $0$ or $1$ vertices from $R(G)$, and $N(K_3,G_R)$ triangles containing $3$ vertices from $R(G)$. It remains to bound the number of triangles which contain exactly one vertex from the $(k-1)|V(F)|$ vertices not in $R(G)$.
		
		Let us consider the subgraph $H$ which consists of the largest number (which is at most $(k-2)$) of vertex-disjoint copies of $F^*$ where each copy has its center in $L(G)$ and all its other vertices in $R(G)$. Let $x$ be a vertex in $L(G)$ that does not belong to $H$ and consider the graph $H_0$ induced by its neighborhood in $R(G)\setminus V(H)$. Observe that $H_0$ cannot contain more than $(k-1)|V(F)|+1$ vertex-disjoint copies of $F_u$. Thus the number of edges in $H_0$ is at most $$ex(n,((k-1)|V(F)|+1)F_u)=ex(n,F_u)+O(n)$$ using Theorem \ref{Gorgol}. In addition the number of edges incident to $R(G)\cap V(H)$ is at most $$(k-2)((k-1)|V(F)|+1)|V(F)|n=O(n).$$ Thus $x$ is in $ex(n,F_u)+O(n)$ triangles. There are at most $(k-1)|V(F)|$ vertices in $L(G)$. Moreover, there are at most $k-2$ vertices in $L(G)$ that belong to $H$, and each of them is in at most $|E(G_R)|$ triangles. Therefore, the number of triangles which contain exactly one vertex from the $(k-1)|V(F)|$ vertices not in $R(G)$ is at most $(k-2)|E(G_R)|+(k-1)|V(F)|ex(n,F_u)+O(n)$.
        
        Therefore, the total number of triangles in $G$ is at most $$N(K_3,G_R)+ (k-2)|E(G_R)|+(k-1)|V(F)|ex(n,F_u)+O(n)$$ $$\le ex^*(n,F)+(k-1)|V(F)|ex(n,F_u)+O(n),$$ completing the proof.
        


	\end{proof}
	
	\begin{rem} 
		
		Note that the proof shows a stronger upper bound. In one of the two cases we get an upper bound matching the lower bound. In the other case we obtain an upper bound of the form $(k-2) \abs{E(G)} + N(K_3,G)$ rather than $(k-1) \abs{E(G)} + N(K_3,G)$ as in the definition of $ex^*(n,F)$. In case $ex(n,F_u)$ has smaller order of magnitude than $ex(n,F)$ and $n$ is large enough, this implies that $ex(n,K_3,kF)= ex^*(n-k+1,F)$.
		
	\end{rem}
	
	Note that if $F_u$ is not a forest, we also know $ex(n,|V(F)|F_u)=(1+o(1))ex(n,F_u)$ by Theorem \ref{Gorgol}. Theorem \ref{2copy} shows that if $ex(n,F_u)=o(ex^*(n,F))$, then we have 
	
	\begin{equation}
		\label{fu}
		ex(n,K_3,kF)=(1+o(1))ex^*(n,F).
	\end{equation}
	This implies the following.
	
	\begin{cor} 
		\begin{enumerate}
			\item[(a)] If $F$ has chromatic number $r > 3$, then we have $$ex(n,K_3,kF)=(1+o(1))\binom{r-1}{3}\left(\frac{n}{r-1}\right)^3.$$
			
			\item[(b)] For $k \ge 1$, we have $$ex(n,K_3,kK_{2,t})=(1+o(1)) \left(\frac{(k-1)(t-1)^{1/2}}{2}+\frac{(t-1)^{3/2}}{6}\right)n^{3/2}.$$
		\end{enumerate}
	\end{cor}
	
	\begin{proof}
		First let us prove (a). For any vertex $u \in V(F)$, trivially $ex(n,F_u) = O(n^2)$. Alon and Shikhelman \cite{ALS2016} showed that $ex(n,K_3,F)=(1+o(1))\binom{r-1}{3}\left(\frac{n}{r-1}\right)^3$ which is equal to $ex^*(n,F)$ asymptotically. Thus $ex(n,F_u)=o(ex^*(n,F))$, so by \eqref{fu}, we are done.
	
    \
		
        Now we prove (b). Alon and Shikhelman \cite{ALS2016} showed that $$ex(n,K_3,K_{2,t})=(1+o(1))\frac{(t-1)^{3/2}}{6}n^{3/2}.$$ In fact, they establish the lower bound $ex(n,K_3,K_{2,t}) \ge (1+o(1))\frac{(t-1)^{3/2}}{6}n^{3/2}$ by considering the $K_{2,t}$-free graph constructed by F\"uredi \cite{Furedik2t} and counting the number of triangles in it. This graph contains $\frac{(t-1)^{1/2}}{2}n^{3/2}(1+o(1)) = ex(n, K_{2,t})$ edges, so it follows that $$ex^*(n,K_{2,t}) \ge (1+o(1)) \left(\frac{(k-1)(t-1)^{1/2}}{2}+\frac{(t-1)^{3/2}}{6}\right)n^{3/2}$$ and this bound is sharp since  by definition, $ex^*(n,K_{2,t}) \le ex(n,K_3,K_{2,t}) + (k-1) ex(n, K_{2,t})$. 
		Now it can be easily seen that $ex(n,F_u)=O(n)=o(ex^*(n,F))$ if $F = K_{2,t}$ for some $u$. So, by \eqref{fu} again, the proof is complete.
		
		
	\end{proof}
	
	\subsection{Counting complete graphs}
	
	\begin{thm}\label{compl1} Let $F$ be a graph and $k,r$ be integers. Let $\max_{m\le r}(ex(n,K_m,F)) = ex(n,K_{m_0},F)$. Then we have $$ex(n,K_r,kF) = O(ex(n,K_{m_0},F)).$$ Moreover, if $k> r-m_0$, then $$ex(n,K_r,kF) = \Theta(ex(n,K_{m_0},F)).$$
		
	\end{thm}
	
	\begin{proof} Let us consider a $kF$-free graph $G$, and its canonical partition. Then every copy of $K_r$ consists of $m$ vertices in $R(G)$ and $r-m$ vertices in $L(G)$ for some integer $m \le k$. These latter ones can be chosen in at most $\binom{k|V(F)|}{r-m}=O(1)$ ways, thus we have $$ex(n,K_r,kF) =O(\sum_{m\le r} ex(n,K_{m},F))\le  O(ex(n,K_{m_0},F)).$$
		
		Let us now consider the graph on $n-r+m_0$ vertices that contains the most copies of $K_{m_0}$, and add $r-m_0$ universal vertices. The resulting graph contains $\Omega(ex(n,K_{m_0},F))$ copies of $K_r$. On the other hand, every copy of $F$ contains at least one of the additional vertices, thus there are at most $r-m_0<k$ pairwise vertex-disjoint copies of $F$ in it.
		
	\end{proof}
	

	
	

	\section{Forbidding complete graphs}
	
	As we mentioned in the introduction, Erd\H{o}s \cite{E1962} determined the exact value of $ex(n,K_s,K_t)$ for $s < t$ and Simonovits \cite{S1966} determined the exact value of $ex(n,kK_t)$ for sufficiently large $n$. 
	
	In this section we investigate the function $ex(n, K_s, kK_t)$. First we present our main result of this section, which determines the order of magnitude for every $s$, $t$ and $k$. Then we show two asymptotic results (for special values of $s$ and $t$).
	
	\subsection{$ex(n, K_s, kK_t)$}
	
	
	The following theorem determines the order of magnitude of $ex(n,K_s,kK_t)$ for all $s,t$ and $k$ (as $n$ tends to infinity). Note that if $s<t$, then the Tur\'an graph shows $ex(n,K_s,kK_t)=\Theta(n^s)$.
	
	\begin{thm}\label{completegen} Let $s \ge t \ge 2$ and $k \ge 1$ be arbitrary integers and let $x:=\lceil \frac{kt-s}{k-1}\rceil-1$. Then we have $$ex(n,K_s,kK_t)=\Theta(n^x).$$
		
	\end{thm}
	
	\begin{proof}
		For the lower bound, consider the Tur\'an graph $K_{s-x} + T_x(n-s+x)$. This graph is $kK_t$-free as $k$ vertex-disjoint copies of $K_t$ together contain at most $kx$ vertices from $T_x(n-s+x)$ and at most $s-x$ vertices from $K_{s-x}$. Thus they together contain at most $s+(k-1)x<kt$ vertices. On the other hand, the Tur\'an-graph $T_x(n-s+x)$ contains $\Omega(n^x)$ copies of $K_x$, and they all can be extended to different copies of $K_s$. 
		
		\vspace{2mm}
		
		\noindent
		To prove the upper bound we will repeatedly apply the canonical partition operation.
		
		\bigskip
		
		\noindent
		\textbf{Step 1:}
		
		\medskip
		\begin{itemize}
			\item[(1.1)] Consider a $kK_t$-free graph $G_1$ and its canonical $(k,K_t)$-partition. 
			
			\smallskip
			
			\item[(1.2)] Let us fix an arbitrary nonempty $X_1 \subset L(G_1)$ (of size $x_1$) and let $$\cA(X_1):=\{A : A \textrm{ is a copy of }K_s \textrm{ in } G_1 \textrm{ with } |V(A) \cap L(G_1)|=X_1 \}.$$ Note that $V(A) \cap R(G_1)$ spans a (copy of) $K_{s-x_1}$ for all $A \in \cA(X_1)$ and let $G_2$ be the subgraph of $G_1$ spanned by the union of $\{V(A) \cap R(G_1) : A \in \cA(X_1) \}$. So $G_2$ is a graph on the vertex set $R(G_1)$. We consider two cases:
		\end{itemize}
		\medskip

		\textbf{Case 1:} $G_2$ contains $k$ disjoint copies of $K_{s-x_1}$. Observe that $s-x_1 < t$ and 
		let us denote the corresponding copies of $K_{s-x_1}$ by $A_1,\ldots,A_k$. We claim that in this case we have $$x_1+k(s-x_1) < kt.$$ Otherwise we could complete $A_1, \ldots , A_k$ from $X_1$ into $k$ disjoint copies of $K_t$ as every vertex of $G_2$ is connected to every vertex of $X_1$ and that would be a contradiction. This inequality implies $s-x_1 \le x$, and as obviously there are $O(n^{s-x_1})$ copies of $K_s$ in $\cA(X_1)$ we stop the application of canonical partitions here, and we are done.
		
		\smallskip
		
		\textbf{Case 2}: $G_2$ does not contain $k$ disjoint copies of $K_{s-x_1}$. In this case we jump to Step 2.
		
		\bigskip
		\noindent
		Now we describe the $i$th step for $i \ge 2$:
		
		\bigskip
		\noindent
		\textbf{Step i:} We have from the $(i-1)$th step: 
		
		\medskip
		
		1) a sequence of subsets $X_1,L(G_1),\ldots,X_{i-1},L(G_{i-1})$ of the vertex set of our initial graph $G_1$, where:
		
		\smallskip
		
		1.1) $X_j \subset L(G_j)$ for all $j \le i-1$, and
		
		\smallskip
		
		1.2) $L(G_j)$ ($j \le i-1$) are pairwise disjoint.
		
		\medskip
		
		2) A set of copies of $K_s$ in $G_1$ parametrized by $X_1,\ldots,X_{i-1}$:
        $$\cA(X_1,\ldots,X_{i-1}):=$$ $$\{A : A \textrm{ is a copy of }K_s \textrm{ in }  G_1 \textrm{ with } |V(A) \cap L(G_1)|=X_1,\ldots,|V(A) \cap L(G_{i-1})|=X_{i-1} \}.$$
		
		3) $G_i$, that is the subgraph of $G_{i-1}$ spanned by the (union of the) edges of the copies in $\cA(X_1,\ldots,X_{i-1})$ on the vertex set $R(G_{i-1})$, and $G_i$ is a $kK_{s-x_1-\ldots-x_{i-1}}$-free graph.
		
		\
		
		We do the following in Step $i$:
		
		\medskip
		\begin{itemize}
			
			\item[(i.1)] We consider the canonical $(k,K_{s-x_1-\ldots-x_{i-1}})$-partition of $G_i$.
			
			\smallskip
			
			\item[(i.2)] We fix an arbitrary nonempty $X_i \subset L(G_i)$ (of size $x_i$) and let $$\cA(X_1,\ldots,X_i):=$$ $$\{A : A \textrm{ is a copy of }K_s \textrm{ in } G_1 \textrm{ with } |V(A) \cap L(G_1)|=X_1,\ldots,|V(A) \cap L(G_i)|=X_i \}.$$ Note that $V(A) \cap R(G_i)$ spans a (copy of) $K_{s-x_1-\ldots-x_i}$ for all $A \in \cA(X_1,\ldots,X_i)$ and let $G_{i+1}$ be the subgraph of $G_i$ spanned by the union of the edges of the elements of $\{V(A) \cap R(G_i) : A \in \cA(X_1,\ldots,X_i) \}$. So $G_{i+1}$ is a graph on the vertex set $R(G_i)$. We consider two cases:
			
		\end{itemize}
		
		\medskip
		
		\textbf{Case 1:} $G_{i+1}$ contains $k$ disjoint copies of $K_{s-x_1-\ldots-x_i}$. Let us denote the corresponding copies of $K_{s-x_1-\ldots-x_i}$ by $A_1,\ldots,A_k$. We claim that in this case we have $$x_1+\ldots+x_i+k(s-x_1-\ldots-x_i) < kt.$$ Otherwise we could complete the sets $V(A_1) \cap R(G_1), \ldots , V(A_k) \cap R(G_2)$ from $X_1\cup\ldots\cup X_i$ into $k$ disjoint copies of $K_t$ as all the vertices of $G_{i+1}$ are connected to all vertices in $X_1\cup\ldots\cup X_i$ and that would be a contradiction. This inequality implies $s-x_1-\ldots-x_i \le x$, and as obviously there are $O(n^{s-x_1-\ldots-x_i})$ copies of $K_s$ in $\cA(X_1,\ldots,X_i)$ we stop the application of canonical partitions here.
		
		\smallskip
		
		\textbf{Case 2}: $G_{i+1}$ does not contain $k$ disjoint copies of $K_{s-x_1-\ldots-x_i}$. In this case we jump to Step ($i+1$).
		
		\
		
		Note that our algorithm finishes in at most $kt$ steps as we always choose nonempty subsets. There are at most $O(1)$ ways to pick $X_1,\ldots,X_s$ and every copy of a $K_s$ will be an element of some $\cA(X_1,\ldots,X_j)$ for some $j \le s$ (and we stop the algorithm in Step $j$), so we are done with the proof.
		
	\end{proof}
	
	\begin{thm} If $t >s$, then we have $$ex(n, K_s, kK_t)=(1+o(1))\binom{t-1}{s}\left ( \frac{n}{t-1} \right )^{s}.$$ 
		
	\end{thm}
	
	\begin{proof} To prove the lower bound one just considers the Tur\'an graph $T_{t-1}(n)$.
		
		For the upper bound consider a $kK_t$-free graph $G$ and its canonical partition. First let us count the copies of $K_s$ that have a common vertex with $L(G)$. There are $O(1)$ ways to pick the vertices from $L(G)$ and $O(n^{s-1})$ ways to pick the remaining vertices from $R(G)$. 
		
		Now let us count those copies that are in $G_R$. As $G_R$ is a $K_t$-free graph on at most $n$ vertices, by Proposition \ref{Erdosprop} there are at most $$(1+o(1))\binom{t-1}{s}\left ( \frac{n}{t-1} \right )^{s}$$ copies of $K_s$ in it. Adding these bounds up, the proof is complete.
	\end{proof}
	
	\begin{thm}\label{compli} If $s \ge t \ge s-k+2$ then we have $$ex(n, K_s, kK_t)=(1+o(1))\binom{k-1}{ s-t+1}\left ( \frac{n}{t-1} \right )^{t-1}.$$
	\end{thm}
	
	\begin{proof}
		For the lower bound consider the graph $K_{k-1}+T_{t-1}(n-k+1)$. Counting the number of $K_s$'s that contain exactly $s-t+1$ vertices from $K_{k-1}$ gives the desired lower bound.
		
		
		For the upper bound consider a $kK_t$-free graph $G$ and its canonical partition. Then any copy of $K_s$ contains at most $t-1$ vertices from $R(G)$. 
		The number of those copies that contain exactly $i$ vertices from $R(G)$ is at most $\binom{t(k-1)}{s-i}\binom{n}{i}=O(n^i)$, thus it is enough to only take care of those copies that contain exactly $t-1$ vertices from $R(G)$. 
		
		Let $$\cA:=\{A : A \textrm{ is a copy of }K_s \textrm{ in } G \textrm{ with } |V(A) \cap R(G)|=t-1 \}.$$ Then we want to upper bound the cardinality of $\cA$. Note that any element of $\cA$ intersects $R(G)$ in $t-1$ vertices spanning a complete graph in $G$. Let $B_1,\dots, B_r$ be the copies of $K_t$ defining $G_L$ with for some $r<k$.
		
		Fix an arbitrary $A_1 \in \cA$ that intersects $B_1$. One can easily see that the cardinality of $$\cA'_1:=\{A \in \cA : V(A)\cap V(B_1) \cap R(G) \neq \emptyset\}$$ is $O(n^{t-2})$. Consider any $A'_1 \in \cA \setminus \cA'_1$ such that $A_{1,R}=V(A)\cap R(G)$ and $A'_{1,R}=V(A'_1)\cap R(G)$ are disjoint and $A'_1 \cap B_1 \neq \emptyset$. The existence of such $A'_1$ implies that $A_1 \cap B_1$ and $A'_1 \cap B_1$ are the same one element since otherwise $B_1$ could be replaced by two disjoint copies of $K_t$ in $G$ (namely those spanned by $V(A_R)\cup \{x\}$ and $V(A'_R)\cup \{y\}$ for some $x \neq y \in B_1$), contradicting the definition of canonical partition. 
		
		In the same way for every $1 < i \le r$ we can fix  $A_i \in \cA$, that intersects $B_i$, and (except a set $\cA'_i \subset \cA$, whose cardinality is $O(n^{t-2})$) we get that for every $A \in \cA \setminus \cA'_i$ we have $$A \cap B_i \subset B_i \cap A_i$$ where $A_i\cap B_i$ contains either $0$ or $1$ vertex.
		Altogether we have that for every $A \in \cA \setminus \cup_{i=1}^{r} \cA'_i$ $$V(A) \cap L(G) \subset \cup_{i=1}^{r} (V(B_i) \cap V(A_i)).$$
		
		By Proposition \ref{Erdosprop} there are at most $(1+o(1))\left ( \frac{n}{t-1} \right )^{t-1}$ copies of a $K_{t-1}$ in a $K_t$-free graph. The statement easily follows.
		
		
	\end{proof}
	
	\section{Forbidding cycles}
	
	The systematic study of counting substructures in $2k$-cycle-free graphs was initiated independently in  \cite{GGMV2017+,GP2017+} and \cite{GS2017+}. 
	
	
	
	
	
	
	\subsection{Counting complete graphs}
	
	For odd cycles, we have the following interesting phenomenon, depending on whether the size of the clique is bigger than $k$ or not.
	
	\begin{thm} 
		\begin{enumerate}
			\item[(a)] If $r \le k$, then $ex(n,K_r,kC_{2l+1}) = \Theta(n^2)$.
			
			\item[(b)] If $r > k+1$, then $ex(n,K_r,kC_{2l+1})=O(n^{1+1/l})$.
		\end{enumerate}
	\end{thm}
	
	\begin{proof} First, let us prove $(a)$. The lower bound follows from Theorem \ref{compl1}, as $$\max_{m\le r}(ex(n,K_m,C_{2l+1}))\ge ex(n,K_2,C_{2l+1})=\Theta(n^2).$$ The quadratic upper bound similarly follows from Theorem \ref{compl1}, using that for any $r \ge 2$, $ex(n, K_r, C_{2l+1})=O(n^2)$ (we used Theorem \ref{GyLemonsthm} for $r \ge 3$).
		
		Now we prove $(b)$. Consider a $kC_{2l+1}$-free graph $G$, and its canonical partition. Then every copy of $K_r$ consists of $m$ vertices in $R(G)$ and $r-m$ vertices in $L(G)$ for some $m$. If $m>2$, then by Theorem \ref{GyLemonsthm}, there are $O(n^{1+1/l})$ copies of a $K_m$ in $G_R$, as it is $C_{2l+1}$-free. Moreover, there are $O(1)$ ways to choose the $r-m$ vertices in $L(G)$, so there are at most $O(n^{1+1/l})$ copies of $K_r$ in $G$ such that $m > 2$. Note that the case $m\le 1$ only gives linearly many copies of $K_r$. Thus we only have to deal with the case $m = 2$. In other words, it only remains to show that the number of copies of $K_r$ in $G$ which contain exactly two vertices from $R(G)$ is $O(n^{1+1/l})$.
		
		Let $G'_R$ be the subgraph of $G_R$ consisting of only those edges $xy \in E(G_R)$ such that $x,y$ and some $r-2$ vertices from $L(G)$ form a copy of $K_r$ in $G$. Clearly, the number of copies of $K_r$ in $G$ which contain exactly two vertices from $R(G)$ is  $O(E(G'_R))$ because each edge of $G'_R$ can be extended to such a copy of $K_r$ in at most $\binom{\abs{L}}{r-2} = O(1)$ ways. 
		Since an edge $xy$ of $G'_R$ appears in a copy of $K_r$ that contains $r-2$ vertices in $L(G)$, $x$ and $y$ have at least $r-2 \ge k$ common neighbors in $L(G)$. If  $G'_R$ contains $k$ vertex-disjoint copies of $C_{2l}$, then we pick an arbitrary edge from each of these $k$ copies. For these $k$ edges $e_1,\dots,e_k$ we can greedily pick $k$  vertices $v_1,\dots, v_k$ in $L(G)$ such that $v_i$ is adjacent to both endpoints of $e_i = x_iy_i$ for every $i$. Then we replace $e_i$ with $v_ix_i$ and $v_iy_i$, and this way we get $k$ vertex disjoint copies of $C_{2l+1}$ in $G$, a contradiction. Thus $G'_R$ does not contain $k$ vertex-disjoint copies of $C_{2l}$, hence $$|E(G'_R)|\le ex(n,kC_{2l})=ex(n,C_{2l})+O(n),$$ by Theorem \ref{Gorgol}. A well-known theorem of Bondy and Simonovits \cite{BS1974} states that $ex(n,C_{2l}) = O(n^{1+1/l})$, so $$|E(G'_R)|\le ex(n,kC_{2l}) = O(n^{1+1/l}).$$ Therefore, the number copies of $K_r$ in $G$ which contain exactly two vertices from $R(G)$ is also $O(n^{1+1/l})$, as desired.
	\end{proof}
	
	For even cycles, we have the following.
	\begin{prop}
		For any $r \ge 2$, $l \ge 2$, we have
		$ex(n,K_r,kC_{2l})=O(n^{1+1/l})$.
	\end{prop}
	\begin{proof}
		Using Theorem \ref{arbit}, we get $ex(n,K_r,kC_{2l}) = O(\overline{ex}(n, K_r,C_{2l}))= O(\sum_{i=1}^r ex(n, K_i,C_{2l}))$, as any induced subgraph of $K_r$ is also a complete graph. By Theorem \ref{GyLemonsthm}, we have $ex(n, K_i,C_{2l}) = O(n^{1+1/l})$ for any $i \ge 1$, so $ex(n,K_r,kC_{2l})=O(n^{1+1/l})$, as required.
	\end{proof}
	
	

	\section{Forbidding bipartite graphs}
	
	Let $K_{a,b}$ denote a complete bipartite graph with color classes of sizes $a$ and $b$ with $a \le b$. Alon and Shikhelman proved the following.
	
	\begin{prop}[\cite{ALS2016}, Proposition 4.10] 
		\label{alon_s_Kab}
		If $s\le t$ and $a\le b <s$ then
		$ex(n,K_{a,b},K_{s,t})=O(n^{a+b-ab/s})$.
		
	\end{prop}
	
	We will prove that the same upper bound holds for $ex(n,K_{a,b},kK_{s,t})$. Note that they also gave a constant factor in their proof; our proof would give a worse constant factor for the case $k>1$. They also showed that in some range of $a$ and $b$ this order of magnitude is sharp, this immediately implies the same for the case $k>1$.
	
	\begin{prop}
		\label{k_abst}
		If $s\le t$ and $a\le b <s$ then
		$ex(n,K_{a,b},kK_{s,t})=O(n^{a+b-ab/s})$.
		
	\end{prop}
	
	\begin{proof} Let $G$ be a $kK_{s,t}$-free graph and consider its canonical partition. A copy of $K_{a,b}$ intersects $G_R$ in a copy of $K_{a',b'}$ for some $0\le a'\le a$ and $0\le b'\le b$. For fixed $a'$ and $b'$, there are $O(n^{a'+b'-a'b'/s})$ such copies by Proposition \ref{alon_s_Kab} since $G_R$ is $K_{s,t}$-free. Increasing $a'$ by $1$, increases  $a'+b'-a'b'/s$ by $1-b'/s$ which is positive since $b' < s$. Applying a similar argument for $a'$, we get that $a'+b'-a'b'/s$ is maximized when $a' = a$ and $b' = b$. Therefore, $O(n^{a'+b'-a'b'/s}) \le O(n^{a+b-ab/s})$.
    
		Moreover, the number of ways to extend a copy of $K_{a',b'}$ to a copy of $K_{a,b}$ by adding vertices from $L(G)$ is at most a constant (where this constant depends on $a,b,s,t$,but not on $n$). Now as $0\le a'\le a$ and $0\le b'\le b$, there are only at most $(a+1)(b+1)$ ways to pick $a'$ and $b'$, finishing the proof.
	\end{proof}
	
	Note that by Theorem \ref{arbit} and Remark \ref{alphaH} we have the following general lower bound: $ex(n,K_{a,b},kK_{s,t}) \ge \Omega(n^{\alpha(K_{a,b})})=\Omega(n^b)$ if $k \ge a+b$ and $a \le b$. If $b>s$, then $a+b-ab/s<b$, so the upper bound in the above proposition cannot hold in this case. Instead, we have the following.
	
	\begin{prop} 
    \label{prop21}
    If $a\le b$, $b\ge s$, $s\le t$, then $ex(n,K_{a,b},kK_{s,t})=O(n^b)$. Moreover, if $k>a$, then we have $ex(n,K_{a,b},kK_{s,t})=\Theta(n^b)$.
		
	\end{prop}
	
	\begin{proof} Let $G$ be a $kK_{s,t}$-free graph and let us consider its canonical partition. A copy of $K_{a,b}$ intersects $G_R$ in a copy of $K_{a',b'}$ for some $0\le a'\le a$ and $0\le b'\le b$ with $a' \le b'$. Let us fix $a'$ and $b'$ and consider two cases. 
		
		If $b'<s$, then by Proposition \ref{alon_s_Kab} there are $O(n^{a'+b'-a'b'/s})$ copies of $K_{a',b'}$ in $G_R$ as it is $K_{s,t}$-free. By the same  argument as in the proof of Proposition \ref{k_abst}, it is easy to see that $a'+b'-a'b'/s$ increases, if we increase $a'$ as long as $b' < s$. So $a'+b'-a'b'/s < s+b'-sb'/s = s \le b$. Thus there are at most $O(n^b)$ copies of $K_{a',b'}$ in $G_R$.
		
		If $b'\ge s$, then we first claim that there are at most $O(n^b)$ copies of $K_{a',b'}$ in $G_R$. Indeed, there are at most $O(n^{b'})$ ways to pick $b'$ vertices, and they have at most $t-1$ common neighbors (otherwise we can find a $K_{s,t}$ in $G_R$). Thus there are at most $\binom{t-1}{a'}=O(1)$ ways to pick the other class of $K_{a',b'}$, so there are at most $O(n^{b'}) \le O(n^b)$ copies of $K_{a',b'}$ in $G_R$ again.
		
		Each copy of $K_{a',b'}$ in $G_R$ can be extended to a copy of $K_{a,b}$ in $G$ by adding vertices from $L(G)$ in $O(1)$ ways. Thus for any fixed $a'$ and $b'$, there are at most $O(n^b)$ copies of $K_{a,b}$ in $G$ and as there are only at most $(a+1)(b+1)$ choices for $a'$ and $b'$, the proof is complete.
		
		
		
		For the moreover part, take a $K_{s,t}$-free graph $G$ on $n-k+1$ vertices and consider $K_{k-1} + G$. It is clearly $k K_{s,t}$-free. Any set of $a$ vertices from $K_{k-1}$ together with any set of $b$ vertices from $G$ forms a copy of $K_{a,b}$, which finishes the proof.
	\end{proof}
	
	Let us focus now on the case $k=1$ and $b \ge s$, as this case was not examined in \cite{ALS2016}.
	\begin{prop}\label{bipartite} 
		\begin{enumerate}
			\item[(a)] If $s\le a \le b \le t$, then $ex(n,K_{a,b},K_{s,t})=O(n^s)$.
			
			\item[(b)] If $a < s \le b\le t$, then $ex(n,K_{a,b},K_{s,t})=\Theta(n^b)$.
		\end{enumerate}
	\end{prop}
	
	\begin{proof} For (a) let us consider a $K_{s,t}$-free graph $G$ and an arbitrary set $S$ of $s$ vertices. We claim that $S$ is contained in the larger color class (i.e., with size $b$) of at most $O(1)$ copies of $K_{a,b}$. Indeed, the vertices of $S$ have at most $t-1$ common neighbors, thus there are $\binom{t-1}{a}$ ways to pick the other side, and they have at most $t-1$ common neighbors, thus there are at most $\binom{t-1-s}{b-s} = O(1)$ copies of $K_{a,b}$ such that their larger color class contains $S$. So there are at most $O(n^s)$ copies of $K_{a,b}$, as desired.
		
		For (b) let us consider the graph $K_{s-1,n-s+1}$. It is easy to see that it contains $\Omega(n^b)$ copies of $K_{a,b}$. The upper bound $O(n^b)$ follows from Proposition \ref{prop21}. 
		
	\end{proof}
	
	\section{Concluding remarks and open problems}
	
	$\bullet$ Our conclusion is that the significant difference between $ex(n,H,F)$ and $ex(n,H,kF)$ mostly comes from the ability to count subgraphs of $H$ due to the universal vertices (vertices of degree $n-1$). A particular example when this is the case is $ex(n,F,kF)$. We did not deal with this especially interesting case, but for complete graphs Theorem \ref{compli} implies $$ex(n,K_t,kK_t)=(k-1+o(1)) \left(\frac{n}{t-1}\right)^{t-1}.$$

	\noindent
	$\bullet$ Another natural direction is to count $lH$ instead of $H$. Here we present two results of this type. 
	
	\begin{prop}
		\label{matching_triangle_free}
		$$ex(n, lK_2, K_3)= \frac{1}{l!}\Big\lfloor \frac{n^2}{4}\Big\rfloor\Big\lfloor \frac{(n-2)^2}{4}\Big\rfloor\dots \Big\lfloor \frac{(n-2l+2)^2}{4}\Big\rfloor.$$
		
	\end{prop}
	
	\begin{proof} 
		The lower bound is given by the complete bipartite graph $K_{\lfloor n/2 \rfloor,\lceil n/2\rceil}$. 
		
		Let $G$ be an triangle free graph on $n$ vertices. To prove the upper bound, we first select an edge $e_1$ from $G$ and then we select an edge $e_2$ disjoint from $e_1$, and then an edge $e_3$ disjoint from both $e_1$ and $e_2$ and so on. By Mantel's theorem, the maximum number of edges in a triangle-free graph is at most $\lfloor n^2/4\rfloor$ so we can pick $e_1 = u_1v_1$ in at most $\lfloor n^2/4\rfloor$ ways. Since the subgraph of $G$ induced by $V(G) \setminus \{u_1, v_1\}$ is also triangle-free, we can pick $e_2$ in at most $\lfloor (n-2)^2/4\rfloor$ ways and $e_3$ in at most $\lfloor (n-4)^2/4\rfloor$ ways (by Mantel's theorem again) and so on, giving a total of $\lfloor n^2/4\rfloor\lfloor (n-2)^2/4\rfloor\dots \lfloor (n-2l+2)^2/4\rfloor$ ordered tuples of $l$ independent edges $(e_1, e_2, \ldots, e_l)$. Since each copy of $lK_2$ is counted $l!$ times, this implies the desired upper bound.
             

	\end{proof}
	
	Using Proposition \ref{matching_triangle_free}, we prove the following asymptotic result.
	
	\begin{thm} 
		Let $l<k$. Then $$ex(n, lK_3, kK_3)=(1+o(1))\binom{k-1}{l} \left(\frac{n^{2}}{4}\right)^l.$$
	\end{thm}
	
	\begin{proof}
		The lower bound is given by the graph $K_{k-1} + K_{\lfloor (n-k+1)/2 \rfloor,\lceil (n-k+1)/2\rceil}$.
		
		
For the upper bound, let $G$ be a $kK_3$-free graph and consider its canonical partition. We say that a triangle is \textit{good} if it has exactly one vertex in $L(G)$. 

We claim that it suffices to only count those copies of $lK_3$ in which every triangle is good. Indeed, no triangle of $lK_3$ has three vertices in $R(G)$, and if any of them has two vertices in $L(G)$, then we can pick at least $l+1$ vertices from $L(G)$ in $O(1)$ ways, and at most $2l-1$ vertices from $R(G)$ in $O(n^{2l-1}) = o(n^{2l})$ ways, which is covered by the error term in the theorem. So from now on we count the number of copies of $lK_3$ in which every triangle is good; we will refer to such a copy of $lK_3$ as a good copy.

We know that the subgraph of $G$ induced by $L(G)$ consists of (maximum possible number of) vertex-disjoint triangles $A_1,\dots, A_r$ for some $r \le k-1$. Let $a_i,b_i,c_i$ be the vertices of $A_i$ for each $1 \le i \le r$. 

A good copy of $lK_3$ can contain only one of the vertices $a_i,b_i,c_i$ for any $1 \le i \le r$, because otherwise we can find more than $r$ vertex-disjoint triangles in $G$, a contradiction. So in order to count the number of good copies of $lK_3$ in $G$, we first pick $l$ of the $r \le k-1$ triangles -say $A_1, A_2, \ldots, A_l$ without loss of generality - from $G_L$ in $\binom{r}{l} \le \binom{k-1}{l}$ ways, and then count the number of good copies of $lK_3$ in which every triangle has a vertex in one of the triangles $A_1, A_2, \ldots, A_l$. Now we bound this latter number.

First let us assume that there are two good copies of $lK_3$ in $G$ which use two different vertices of the triangle $A_i = a_ib_ic_i$, say $a_i$ and $b_i$ for some $1 \le i \le l$. Let the corresponding good triangles of these $lK_3$'s be $a_ixy$ and $b_ipq$. Then the edges $xy$ and $pq$ must share a vertex, because otherwise we can replace $A_i$ with the triangles $a_ixy$ and $b_ipq$ to produce more than $r$ vertex-disjoint triangles in $G$, a contradiction. Thus it is easy to see that number of good triangles containing $a_i$ is at most $2n$ since there are at most $2n$ edges that share a vertex with $pq$. Similarly, the number of good triangles containing $b_i$ or $c_i$ is also at most $2n$ each. Therefore, the total number of good triangles which have a vertex in $A_i$ is at most $6n = O(n)$ in this case. This implies that the number of good copies of $lK_3$ in which every triangle has a vertex in one of the triangles $A_1, A_2, \ldots, A_l$ is at most $O(n) \cdot O(n^{2l-2}) = o(n^{2l})$, which is covered by the error term of the theorem again.

So we can assume that every such good copy of $lK_3$ contains only one of the vertices $a_i,b_i,c_i$ for each $1 \le i \le l$, say $u_1, u_2, \ldots, u_l$. Thus we can count the number of those copies by picking a copy of $lK_2$ from $G_R$ in at most $ex(n, lK_2, K_3) = (1+o(1)) \frac{1}{l!} \left(\frac{n^{2}}{4}\right)$ ways by Proposition \ref{matching_triangle_free} (recall that $G_R$ is triangle-free), and then pairing the $l$ edges of $lK_2$ with $u_1, u_2, \ldots, u_l$ in at most $l!$ ways. 

Therefore, the total number of good copies of $lK_3$ is at most $(1+o(1))  \binom{k-1}{l} \left(\frac{n^{2}}{4}\right)^l$, as required.

\end{proof}

	
	\vspace{2mm}
	\noindent
	$\bullet$ We mention some more specific open problems.
	
	\medskip
	
	$\circ$ A lower bound of $\Omega(n^{s})$ is trivial in Proposition \ref{bipartite} for $s=1$. However it would be appealing to prove it for all $s$ or even in case $s=2$.
	
	\smallskip
	
	$\circ$ It would be also interesting to improve Theorem \ref{completegen} and prove an asymptotic result.
	
	\smallskip
	
	$\circ$ In this article our results mostly obtain the order of magnitude or asymptotics of various quantities. It would be interesting to prove exact results corresponding to them.
	
	\vspace{2mm}
	\noindent
	$\bullet$ Finally, let us mention that the Tur\'an number of the disjoint union of graphs $F_1,F_2,...,F_k$, has not been investigated when the $F_i$'s can be different. (See Theorem \ref{Gorgol} and the comment after it, for the case when all the $F_i$'s are the same.) It is not hard to prove the following proposition. However, it would be interesting to prove a sharper result in this case.
	
	\begin{prop} \label{gorgol_should_have_done_it} Let us suppose that we have graphs $F_1,...,F_k$ and let $F=\cup_{1\le i \le k}F_i$. Then we have
		$$ex(n,F)=\max \{ex(n,F_i): i\le l\} + O(n).$$
	\end{prop}

	\begin{proof}[Proof of Proposition \ref{gorgol_should_have_done_it}] 
    Let $j$ be an integer such that $ex(n,F_j) = \max \{ex(n,F_i): i\le l\}$. Then the lower bound follows by taking an $F_j$-free graph with maximum possible number of edges. 
    
    For the upper bound, consider an $F$-free graph $G$.
		Let $F'$ be a subgraph of $G$ consisting of vertex disjoint copies of $F_1, \dots, F_j$ where $j$ is an integer which is chosen as large as possible. Clearly $j < l$ as $G$ is $F$-free. Then, of course, the subgraph of $G$ induced by $V(G) \setminus  V(F')$ is $F_{j+1}$-free, so it contains at most $ex(n,F_{j+1}) \le \max \{ex(n,F_i): i\le l\}$ edges. 
       Moreover, there are at most $O(n)$ edges incident to the vertices of $F'$. Adding these bounds up, the proof is complete.
	\end{proof}


	
	\section*{Acknowledgements}
	We are grateful to an anonymous reviewer for very carefully reading our paper and for many helpful comments. 
    
	Research of Gerbner was supported by the J\'anos Bolyai Research Fellowship of the Hungarian Academy of Sciences and by the National Research, Development and Innovation Office -- NKFIH, grant K 116769.
	
	\noindent
	Research of Methuku was supported by the National Research, Development and Innovation Office -- NKFIH, grant K 116769.
	
	\noindent
	Research of Vizer was supported by the National Research, Development and Innovation Office -- NKFIH, grant SNN 116095.

\end{document}